\documentclass{article}
\usepackage{amsmath,amsthm,amssymb,amsfonts,amscd, array}
\usepackage{enumerate}
\usepackage[autostyle]{csquotes}

\usepackage{graphicx,epstopdf,epsfig}
\usepackage{amsfonts,epsfig,fancyhdr,graphics,amsmath,amssymb}
\usepackage{tikz}
\usetikzlibrary{arrows, shapes}
\usepackage{graphicx}
\usetikzlibrary{positioning}

\theoremstyle{theorem}
\newtheorem{theorem}{Theorem}
 \newtheorem{lemma}{Lemma}
 \newtheorem{corollary}{Corollary}
\newtheorem{proposition}{Proposition}
\newtheorem{example}{Example}
 
\theoremstyle{definition}

\begin{document}

\title{The Sign Symmetric $P_{0,1}^+$-Matrix Completion Problem}
\markright{The Sign Symmetric $P_{0,1}^+$-Matrix Completion Problem}
\author{Victor Tomno and Linety Muhati}

\maketitle
\begin{abstract}
We study sign symmetric $P_{0,1}^+$-matrix completion problem. It is shown that any non-asymmetric incomplete digraph lacks sign symmetric $P_{0,1}^+$-completion, digraphs of order at most four are completely classified and finally comparisons between sign symmetric $P_{0,1}^+$-completion and other matrix completions was given.
\end{abstract}



\section{Introduction} \label{intro-sec}
A real $n \times n$ matrix $A$ is {\em $P$-matrix ($P_0$-matrix)} if every principal minor of the matrix is positive (nonnegative)\cite{P0matrix}. A real $n \times n$ matrix $A$ is a $P_{0}^+$-matrix if for $k\in \{1,\dots,n\}$, every $k\times k$ principal minor of $A$ is nonnegative, and at least one $k\times k$ principal minor is positive \cite{po+matrices}. A {\em$P_{0,1}^+$-matrix} is $P_0^+$-matrix with positive diagonal entries \cite{po1+matrices} and a {\em sign symmetric $P_{0,1}^+$-matrix} is a $P_{0,1}^+$-matrix in which the product of twin entries is positive or both are zeros, that is $a_{ij}a_{ji}>0$ or $a_{ij}=a_{ji}=0$.

A {\em partial matrix} is a matrix with some specified entries and others are not specified. We use following proposition to give three cases in which a partial sign symmetric $P_{0,1}^+$-matrix can exist. 
\begin{proposition}
	A partial matrix $A$ is a partial sign symmetric $P_{0,1}^+$-matrix if and only if exactly one of the following holds:
	
	\begin{enumerate}[(i)]
		\item At least one diagonal entry is unspecified, product of fully specified twin entries is positive or both are zeros and each fully specified principal minor of $A$ is nonnegative.
		\item All diagonal entries are specified and are positive, at least one non-diagonal entry is unspecified, product of fully specified twin entries is positive or both are zeros and each fully specified principal minor of $A$ is nonnegative.
		\item All entries are specified and $A$ is a sign symmetric $P_{0,1}^+$-matrix.
	\end{enumerate}
	
\end{proposition}

In the recent past, there have been various researches concerning matrix completions on $P$- matrices, see (\cite{symmetricP}-\cite{wssP}) and $P_{0}$-matrices, see (\cite{wssP}-\cite{nonnegativeP0}). In 2015 Sarma and Sinha \cite{po+matrices} studied a new class on \enquote{The $P_0^+$-matrix completion problem}. They discussed the necessary and sufficient conditions for a digraph to have $P_0^+$-completion and also singled out digraphs of order at most four having $P_0^+$-completion. Two years later, Sinha \cite{po1+matrices} studied on \enquote{The $P_{0,1}^+$-matrix completion problem}. This paper is motivated by the work of Sinha \cite{po1+matrices} and we are interested in the property of \enquote{sign symmetric} i.e., when the product of twin entries is positive or both twin entries are zeros, that is $a_{ij}a_{ji}>0$ or $a_{ij}=a_{ji}=0$. Our purpose here is to study  \enquote{The sign symmetric $P_{0,1}^+$-matrix completion problem}.

A partial matrix $A$ is said to {\em specify a digraph $D$} if position $a_{ij}$ of $A$ is specified if and only if there is an arc between vertex $v_{i}$ and $v_{j}$ of $D$.

A \emph{digraph  $D=(V_{D},E_{D})$} is a finite non-empty set of positive integers $V_{D}$, whose members are called \emph{vertices} and a set, $E_{D}$, of (ordered) pairs $(v_i,v_j)$ of vertices called the \emph{arc} of D. The \emph{order} of a digraph $D$, denoted by $|D|$, is the number of vertices of $D$. A digraph is \emph{complete} if it includes all possible arcs between its vertices, and is denoted by $K_{n}$, where $n$ is the number of vertices. A complete (di)graph is called a \textit{clique}.

A digraph $D=(V_{D},E_{D})$ is $\textit{isomorphic}$ to the digraph $D^{'}=(V_{D^{'}},E_{D^{'}}) $ if there is a bijection map $\phi $ :$V_{D}$ $\rightarrow$ $V_{D^{'}}$ such that $(v,w) $ $\in $ $E_{D} $ if and only if ($\phi(v),\phi(w)$) $\in$ $E_{D^{'}}$. A \emph{symmetric digraph} is a digraph with the property that $(v_{i},v_{j})$ is an arc if and only if $(v_{j},v_{i})$ is an arc.	An \emph{asymmetric digraph} is a digraph with the property that $(v_{i},v_{j})$ is an arc, then $(v_{j},v_{i})$ is not an arc. 

A digraph $H=(V_{H},E_{H}) $ is a  \emph{sub-digraph} of digraph D if $V_{H}$ $ \subseteq $ $V_{D} $ and $E_{H}$ $\subseteq$ $E_{D}$ \cite{graphtheory}. A \emph{cycle} $C$ in a digraph $D=(V_{D},E_{D})$ is a sub-digraph $(V_{C},E_{C})$ where $V_{C}=\{v_{1},\dots,v_{k}\}$ and $E_{C}=\{ (v_{1},v_{2}),(v_{2},v_{3}),\dots,$ $(v_{k-1},v_{k}),(v_{k},v_{1})\}$ and in this case the \emph{length} of $C$ is $k$ and we call $C$ a $k$-cycle. A 
\emph{loop} is an arc with same initial and terminal vertex i.e., $x=(v_{i},v_{i})$. 

A {\em completion} of a partial matrix is a specific choice of values for the unspecified entries resulting in a required matrix. A (di)graph has {\em sign symmetric $P_{0,1}^+$-completion} if every partial sign symmetric $P_{0,1}^+$-matrix which specifies the (di)graph can be completed to a sign symmetric $P_{0,1}^+$-matrix.

\section{Sign symmetric $P_{0,1}^+$-matrix completion}\label{ss$P_{0,1}^+$-completion} In this section we would like to point out the completion of partial sign symmetric $P_{0,1}^+$-matrices when all diagonal entries are unspecified, some are specified and when all are specified.

It is clear that $1 \times 1$ partial matrix has sign symmetric $P_{0,1}^+$-completion since if it is not specified then a positive value is assigned otherwise it is complete cause the entry is a diagonal entry and by the definition of sign symmetric $P_{0,1}^+$-matrix it is a positive value.

Throughout this section, notations for a partial matrix $A$ are given as follows: diagonal entry at position $(i,i)$ as $d_{i}$, specified non-diagonal entry at position $(i,j)$ as $a_{ij}$, unspecified non-diagonal entry at position $(i,j)$ as $x_{ij}$ and $A_c$ is the completion of a partial matrix $A$.  

We start our discussion for the case where all diagonal entries are specified and are positive, at least one non-diagonal entry is unspecified, product of fully specified twin entries is positive or both are zeros and each fully specified principal minor is nonnegative, results given in Theorem \ref{T: hascompletion}.  

 \begin{theorem} \label{T: hascompletion} A digraph which does not include any loop has sign symmetric $P_{0,1}^+$-completion.
 \end{theorem}
 
 \begin{proof}
 	Let $A$ be a $n \times n$ partial sign symmetric $P_{0,1}^+$-matrix. If $a_{ij}$ is specified with $a_{ij} \ne0$ and $x_{ji}$ is not specified then assign $x_{ji}$ an absolutely small value such that $a_{ij}x_{ji}>0$, and if $a_{ij}=0$ then assign $x_{ji}=0$. Next, set sufficiently large $t$ for all diagonal. Any $k\times k$ principal minors will be of the form $t^{k}+ p(t)$, where $p(t)$ is a polynomial of degree less or equal to $k-1$, so for large enough $t$, all $k\times k$ principal minors of a completed matrix $A_c$ are positive for $k=1,\dots,n$. Thus, every partial sign symmetric $P_{0,1}^+$-matrix $A$ that omits all diagonal entries can be completed to a sign symmetric $P_{0,1}^+$-matrix.
 \end{proof}

The following example indicates that, it would be hard or even impossible to complete a partial sign symmetric $P_{0,1}^+$-matrix when some diagonal and non-diagonal entries are specified.
\newline
\newline
\begin{example}\label{example}\normalfont
	Consider the digraph $D_1$ of order 3 having 2 loops.  
		\begin{center}	
			\begin{figure}[th]	
				\hspace{5.5cm}
				\begin{tikzpicture}[scale = 1.35]
				\tikzstyle{m1}=[draw, circle, fill=black, inner sep=2bp];
				\node[m1] (a) at (4,0) {} edge [loop below] ();
				\node at (3.8, 0) {$1$};
				\node[m1] (b) at (4.5,0.7) {}edge [loop above] ();
				\node at (4.7, 0.9, 0) {$2$};
				\node[m1] (c) at (5,0) {}; 
				\node at (5.2, 0) {$3$};

				\draw[<->,>=triangle 45] (a) -- (b);
				\draw[<->,>=triangle 45] (b) -- (c);
				\draw[<->,>=triangle 45] (c) -- (a);
				\end{tikzpicture}
				\caption{Digraph $D_1$ that does not have sign symmetric $P_{0,1}^+$-completion}\label{f:digraphs $D_3(4,1)$} 
			\end{figure}
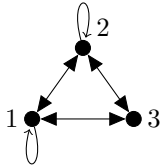
		\end{center}
		
	However the partial sign symmetric $P_{0,1}^+$-matrix $A= \begin{bmatrix}
	1	& 1 & 1 \\
	1	& 1 & 1  \\
	1	& 1 & d_3  
	\end{bmatrix}$ specifying $D_1$ does not have have a sign symmetric $P_{0,1}^+$-completion because any completion of $A$ the first two columns are linearly independent and so for any value of $d_3$, $\det A=0$. 
\end{example}

Partial matrices to be considered next are those with all diagonal entries specified. The following lemma shows that any incomplete digraph having all loops and a 2-cycle lacks sign symmetric $P_{0,1}^+$-completion.
\begin{lemma} \label{2-cycle lacks sign symmetric $P_{0,1}^+$-completion}
	Any incomplete digraph having all loops and a 2-cycle do not have sign symmetric $P_{0,1}^+$-completion. 
\end{lemma}

\begin{proof}
	The proof is by induction.\\
	First, for a digraph of order 2, if it contains 2-cycle then is a complete digraph.\\
	We start with an incomplete digraph of order 3 having all loops and a 2-cycle.
	
	Consider the digraph $D_3(2,1)$ in Figure \ref{f:digraphs $D_3(2,1)$}. 
			\begin{center}	
				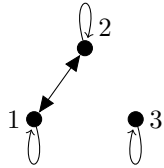
\begin{figure}[th]	
					\hspace{5.5cm}
					\begin{tikzpicture}[scale = 1.35]
					\tikzstyle{m1}=[draw, circle, fill=black, inner sep=2bp];
					\node[m1] (a) at (4,0) {} edge [loop below] ();
					\node at (3.8, 0) {$1$};
					\node[m1] (b) at (4.5,0.7) {} edge [loop above] ();
					\node at (4.7, 0.9, 0) {$2$};
					\node[m1] (c) at (5,0) {} edge [loop below] ();
					\node at (5.2, 0) {$3$};

					\draw[<->,>=triangle 45] (a) -- (b);
					\end{tikzpicture}
					\caption{A Digraph $D_3(2,1)$}\label{f:digraphs $D_3(2,1)$} 
				\end{figure}
			\end{center}

	Let $A= \begin{bmatrix}
	1	& -1 & x_{13} \\
	-1	& 1 & x_{23}  \\
	x_{31}	& x_{32} & 1 
	\end{bmatrix}$ be partial sign symmetric $P_{0,1}^+$-matrix specifying incomplete digraph $D_3(2,1)$ (For more information on digraphs nomenclature, see the book by Harary \cite{graph}). 
	
	We compute the determinant of $A$, $\det A=-x_{13}x_{31}-x_{13}x_{32}-x_{23}x_{31}-x_{23}x_{32}$ which is impossible to assign values to all unspecified entries such that $\det A>0$ and $x_{ij}x_{ji}>0$ or $x_{ij}=x_{ji}=0$, and therefore $D_3(2,1)$ do not have sign symmetric $P_{0,1}^+$-completion. For any other 2-cycle digraph of order 3, we assume any additional specified entry is equal to 0 and hence making the digraph lack sign symmetric $P_{0,1}^+$-completion. 
	
	Next, we consider a 2-cycle digraph of order 4.
		\begin{center}	
			\begin{figure}[th]	
				\hspace{5.5cm}
				\begin{tikzpicture}[scale = 1.35]
				\tikzstyle{m1}=[draw, circle, fill=black, inner sep=2bp];
				\node[m1] (a) at (9,0) {} edge [loop below] ();
				\node at (8.8, -0.2) {$1$};
				\node[m1] (b) at (9,1) {} edge [loop above] ();
				\node at (8.8, 1.2) {$2$};
				\node[m1] (c) at (10,1) {} edge [loop above] ();
				\node at (10.2, 1.2) {$3$};
				\node[m1] (d) at (10,0) {} edge [loop below] ();
				\node at (10.2, -0.2) {$4$};
				
				\draw[<->,>=triangle 45] (a) -- (b);

				\end{tikzpicture}
				\caption{Digraph $D_4(2,1)$}\label{f: D_4(2,1)}
			\end{figure}
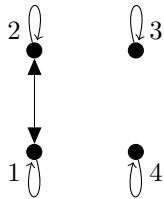
		\end{center}
		
	The partial sign symmetric $P_{0,1}^+$-matrix $A= \begin{bmatrix}
	1	& -1 & x_{13} & x_{14}\\
	-1	& 1 & x_{23} & x_{24} \\
	x_{31}	& x_{32} & 1 & x_{34} \\
	x_{41}	& x_{42} & x_{43} & 1 
	\end{bmatrix}$ 	specifies incomplete digraph $D_4(2,1)$ in Figure \ref{f: D_4(2,1)}. 
	
	Let us first try to complete the partial principal sub-matrix $A(1,2,3)$.
	
	$\det A(1,2,3)=-x_{13}x_{31}-x_{13}x_{32}-x_{23}x_{31}-x_{23}x_{32}$, since every principal minors needs to be nonnegative then the only possible values of the unspecified entries of $A(1,2,3)$ are zeros and thus yielding the partial matrix below
	
	$$A= \begin{bmatrix}
	1	& -1 & 0 & x_{14}\\
	-1	& 1 & 0 & x_{24} \\
	0	& 0 & 1 & x_{34} \\
	x_{41}	& x_{42} & x_{43} & 1 
	\end{bmatrix}$$
	
	Again, we consider the partial principal sub-matrix  $A(1,2,4)$, its determinant is $\det A(1,2,4)=-x_{14}x_{41}-x_{14}x_{42}-x_{24}x_{41}-x_{24}x_{42}$, similarly, for it to be nonnegative, we need to set $x_{14}=x_{24}=x_{41}=x_{42}=0$ and we obtain the partial matrix below.

	$$A= \begin{bmatrix}
	1	& -1 & 0 & 0\\
	-1	& 1 & 0 &0 \\
	0	& 0 & 1 & x_{34} \\
	0	& 0 & x_{43} & 1 
	\end{bmatrix}$$

The matrix above has $\det A=0$ for any values of $x_{34}$ and $x_{43}$, hence the partial sign symmetric $P_{0,1}^+$-matrix $A$ cannot be completed to a sign symmetric $P_{0,1}^+$-matrix.

Also, using row reduction process, row one plus row two to gives row 2 causing the second row to have zeros and so, by the properties of determinants, the determinant of $A$ is zero i.e., $\det A=0$. 

By same argument, for a $n\times n$ partial sign symmetric $P_{0,1}^+$-matrix $A$ specifying a digraphs of order $n$ with 2-cycle in $v_i$ and $v_j$ for $i<j<n$, where $n>4$. Assuming $a_{ij}=a_{ji}=-1$ forces $\det A(i,j,n)=0$ for the choices of $x_{in}=x_{jn}=x_{ni}=x_{nj}=0$ otherwise $\det A(i,j,n)<0$. By row reduction and properties of determinants we get $\det A=0$. Whenever we have more arcs the we assume addition specified  entries are zeros, making the determinant to be always zero. Therefore, any incomplete digraph containing 2-cycle lacks sign symmetric $P_{0,1}^+$-completion.

\end{proof}

We now collect some results from the literature that will be used in the next section.

\begin{lemma} \label{ssP0+l}(\cite{relationships}, Corollary 3.1)
	Any digraph that has sign symmetric $P_{0,1}^+$-completion also has sign symmetric $P$-completion.
\end{lemma}

\begin{lemma} \label{assP0+l}(\cite{relationships}, Corollary 5.2)
	Any asymmetric digraph that has sign symmetric $P$-completion also has sign symmetric $P_{0,1}^+$-completion.
\end{lemma}

The results obtained from this section and the available literature are now enough for complete classifications of order at most four digraphs having sign symmetric $P_{0,1}^+$-completion. 
\section{Classifications of order at most four digraphs having sign symmetric $P_{0,1}^+$-completion}\label{classification}

In this section we give complete classifications of digraphs having order at most four as sign symmetric $P_{0,1}^+$-completion or not. Digraphs will be denoted $D_{p}(q,n)$, where $p$ is the number of vertices, $q$ is the number of arcs and $n$ is the digraph number, i.e., this helps in distinguishing digraphs with same vertices and arcs which are identified, the reader can consult the nice book \cite{graph}. 
\begin{theorem} \label{hasssp01+completion}
	The digraphs $D_{p}(q,n)$ for $1\le p \le 4$ listed below have sign symmetric $P_{0,1}^+$-completion.
	\begin{align*}  
	p &= 1; &\quad  & &\quad  &\\
	p &= 2, &\quad q &= 0,1,2; &\quad  &\\
	p &= 3, &\quad q &= 0,1,6; &\quad  &\\
	& &\quad q &= 2, &\quad n &= 2-4;\\
	& &\quad q &= 3, &\quad n &= 3;\\
	p &= 4, &\quad q &= 0,1,12; &\quad  &\\ 
	& &\quad q &= 2, &\quad n &= 2-5;\\ 
	& &\quad q &= 3, &\quad n &= 4-11,13;\\
	& &\quad q &= 4, &\quad n &= 16-19,21-23,25-27;\\
	& &\quad q &= 5, &\quad n &= 29,31,33,34,36,37;\\ 
	& &\quad q &= 6, &\quad n &= 46.\\
	\end{align*}
\end{theorem}

\begin{proof}
	Part 1: Digraphs that have sign symmetric $P_{0,1}^+$-completion.
	
	First, digraphs $D_{p}(0,1)$ for $1\le p\le4$ have sign symmetric $P_{0,1}^+$-completion since every partial matrix specifying these digraphs (null graphs) can be completed by assigning all unspecified entries to be zeros.
	
	It is clear that $D_{1}(0,1)$, $D_{2}(2,1)$, $D_{3}(6,1)$ and $D_{4}(12,1)$ have sign symmetric $P_{0,1}^+$-completion since matrices specifying digraphs are complete by definition of a partial sign symmetric $P_{0,1}^+$-matrix. 
	
	The following digraphs have sign symmetric $P$-completion [\cite{wssP}, Theorem 4.4] and they are asymmetric digraphs. By Lemma \ref{assP0+l}, they also have sign symmetric $P_{0,1}^+$-completion
	\begin{align*}  
	p &= 2, &\quad q &= 1; &\quad  &\\
	p &= 3, &\quad q &= 1; &\quad  &\\
	& &\quad q &= 2, &\quad n &= 2-4;\\
	& &\quad q &= 3, &\quad n &= 3;\\
	p &= 4, &\quad q &= 1; &\quad  &\\ 
	& &\quad q &= 2, &\quad n &= 2-5;\\ 
	& &\quad q &= 3, &\quad n &= 4-11,13;\\
	& &\quad q &= 4, &\quad n &= 16-19,21-23,25-27;\\
	& &\quad q &= 5, &\quad n &= 29,31,33,34,36,37;\\ 
	& &\quad q &= 6, &\quad n &= 46.\\
	\end{align*}

	Part 2: Digraphs that do not have sign symmetric $P_{0,1}^+$-completion.
	
	The following are incomplete digraphs having contains a 2-cycle and according to Lemma \ref{2-cycle lacks sign symmetric $P_{0,1}^+$-completion}, they do not have sign symmetric $P_{0,1}^+$-completion
	\begin{align*}  
	p &= 3, &\quad q &= 2 &\quad n &= 1;\\
	& &\quad q &= 3, &\quad n &= 1,4;\\
	& &\quad q &= 4,5; &\quad  &\\
	p &= 4, &\quad q &= 2, &\quad n &=1\\ 
	& &\quad q &= 3, &\quad n &= 1-3;\\ 
	& &\quad q &= 4, &\quad n &= 1-15;\\
	& &\quad q &= 5, &\quad n &= 1-28;\\
	& &\quad q &= 6, &\quad n &= 1-44;\\ 
	& &\quad q &= 7,8,9,10,11. &\quad  &\\
	\end{align*}
	
	The following digraphs do not have sign symmetric $P$-completion [\cite{wssP}, Theorem 4.4], and hence by Lemma \ref{ssP0+l}, they do not have sign symmetric $P_{0,1}^+$-completion (this list does not include those that fall under this rule but were already listed in a previous list):
	\begin{align*}  
	p &= 3, &\quad q &= 3 &\quad n &= 2;\\
	p &= 4, &\quad q &= 3, &\quad n &=12\\ 
	& &\quad q &= 4, &\quad n &= 20,24;\\
	& &\quad q &= 5, &\quad n &= 30,32,35,38;\\
	& &\quad q &= 6, &\quad n &= 45,47,48.\\ 
	\end{align*}

\end{proof}

\section{Comparisons between sign symmetric $P_{0,1}^+$-completion and other matrix completions} \label{comparison ssP_{0,1}^+-completion}
In this section we show comparisons between sign symmetric $P_{0,1}^+$-completion and other matrix completions: sign symmetric $P$-completion, sign symmetric $P_{0,1}$-completion, sign symmetric $P_0^+$-completion and sign symmetric $P_0$-completion.

First,  sign symmetric $P$-matrices are sign symmetric $P_{0,1}^+$-matrices, sign symmetric $P_{0,1}^+$-matrices are sign symmetric $P_{0,1}$-matrices and sign symmetric $P_{0,1}$-matrices are sign symmetric $P_0$-matrices.

Second, sign symmetric $P$-matrices are sign symmetric $P_{0,1}^+$-matrices, sign symmetric $P_{0,1}^+$-matrices are sign symmetric $P_0^+$-matrices and sign symmetric $P_0^+$-matrices are sign symmetric $P_0$-matrices.

We have already given some relationships in Lemmas \ref{ssP0+l} and \ref{assP0+l}, and we are giving more relationships here.

\begin{lemma} \label{symetricp}(\cite{relatedclasses}, Theorem 2.12)
	Any digraph that has sign symmetric $P_{0,1}$-completion also has sign symmetric $P$-completion.
\end{lemma}

\begin{lemma} \label{assymetricp}(\cite{relatedclasses}, Theorem 2.13)
	Any asymmetric digraph that has sign symmetric $P$-completion also has sign symmetric $P_{0,1}$-completion.
\end{lemma}

\begin{lemma} \label{ssp+01}(\cite{relationships}, Corollary 4.3)
	Any digraph that has sign symmetric $P_0^+$-completion also has sign symmetric $P_{0,1}^+$-completion.
\end{lemma}

\begin{theorem} \label{assP_0^+ completion}
Any digraphs that has sign symmetric $P_{0,1}^+$-completion also has sign symmetric $P_{0,1}$-completion.	
\end{theorem}

\begin{proof}
By Lemma \ref{ssP0+l}, digraphs that have sign symmetric $P_{0,1}^+$-completion also have sign symmetric $P$-completion
 and according to Lemma \ref{2-cycle lacks sign symmetric $P_{0,1}^+$-completion}, all non-asymmetric digraphs do not have sign symmetric $P_{0,1}^+$-completion and therefore, they are asymmetric digraphs that have sign symmetric $P$-completion. Thus, by Lemma \ref{assymetricp}, they also have sign symmetric $P_{0,1}$-completion. 
\end{proof}
As a consequence of Theorem \ref{assP_0^+ completion}, we have the following corollaries. 
\begin{corollary}
Any digraph having sign symmetric $P_0^+$-completion also has sign symmetric $P_{0,1}$-completion.
\end{corollary}

\begin{corollary}
	Any digraph having sign symmetric $P_0^+$-completion also has sign symmetric $P$-completion.
\end{corollary}

\begin{lemma} \label{ssP_0 completion}
	Any digraph having sign symmetric $P_0$-completion are null and complete digraphs
\end{lemma}

\begin{proof}
Let $n\times n$ sign symmetric $P_0$-matrix $A_c$ be a completion of partial sign symmetric $P_0$-matrix $A$ having all diagonal entries specified i.e., must be a zero or a positive value. When the partial sign symmetric $P_0$-matrix $A$ specifies a null graph then it can be completed by zero completion, i.e, assigning all non-diagonal entries zeros. Now, when all specified diagonal entries are zeros and specified entry $a_{ij}\ne0$ then it is impossible to assign a value to $x_{ji}$ such that $\det A(i,j) \ge0$. This shows that a partial $n\times n$ sign symmetric $P_0$-matrix with some specified and others unspecified entries lacks sign symmetric $P_0$-completion and so, the only digraphs having sign symmetric $P_0$-completion are null and complete digraphs.

\end{proof}

From Lemma \ref{ssP_0 completion} , results in the next theorem will now be trivial. 

\begin{theorem}
	Any digraphs that has sign symmetric $P_0$-completion also has sign symmetric $P_{0,1}^+$-completion.	
	\end{theorem}

\section*{Acknowledgement}
The authors are very grateful to the anonymous referees and editors for their careful reading of the paper.

\bigskip 
 {\small\rm\baselineskip=10pt
	\baselineskip=12pt
	\qquad Victor Tomno\par
	\qquad \textit{Email address:} {\tt victomno@gmail.com}
	
	\smallskip
		\bigskip
	\qquad Linety Muhati\par
	\qquad \textit{Email address:}{\tt linetnaswa@gyahoo.com}
	
	\bigskip \smallskip
	\qquad Department of Mathematics and Computer Science\par
	\qquad University of Eldoret\par
	\qquad P.O.~Box 30100 - 1125, Eldoret, Kenya\par
	
}
\end{document}